\documentclass[a4paper]{article}

\usepackage[cp1251]{inputenc}
\usepackage[T2A]{fontenc}

\usepackage[english]{babel}
\usepackage{amsmath, amsfonts, amssymb, amsthm}

\theoremstyle{theorem}
\newtheorem{theorem}{Theorem}

\newtheorem*{ponc}{Poncelet's theorem}
\newtheorem*{morl}{Morley's theorem}

\theoremstyle{definition}

\newtheorem*{remark}{Remark}

\usepackage{pgf,tikz}
\usetikzlibrary{arrows}

\numberwithin{equation}{section}
\numberwithin{figure}{section}

\begin{document}

\title{Conics associated with triangles, or how Poncelet meets Morley}

\author{Kostiantyn Drach}

\date{}

\maketitle

\begin{abstract}
We present a criterion when six points chosen on the sides of a triangle belong to the same conic. Using this tool we show how the two geometrical gems~-- celebrated Poncelet's theorem of projective geometry and incredible Morley's theorem of Euclidean geometry~-- can meet in a one construction. 
\end{abstract}

\section{Introductions}

Geometry, although being one of the oldest mathematical sciences, has these ``immortal'' results that keep attracting researchers for centuries from more and more perspectives. One of such results is well-known \textit{Poncelet's theorem}, or \textit{Poncelet's porism} (for classical facts and resent advances on the topic see, for example,~\cite{BB96}, \cite{Ber87}, \cite{DrR11}, \cite{Dui10}, \cite{Fl09}, \cite{LT07}, and references therein). Discovered by Jean-Victor Poncelet in 1813 while being a military prisoner in Saratov, this theorem states the following.

\begin{ponc}
Let $\mathcal C_1$ and $\mathcal C_2$ be two plane conics. If there exist a \emph{Poncelet chain} $P_1P_2\ldots P_n \ldots$, that is a broken line whose vertexes $P_1, P_2, \ldots$ lie on $\mathcal C_1$ and whose links $P_1P_2, P_2P_3, \ldots$, or their continuations, touch $\mathcal C_2$, that closes up in $n$ steps, that is $P_1 = P_{n+1}$, then for any starting point $P_1$ on $\mathcal C_1$ the corresponding Poncelet chain will close up in $n$ steps.
\end{ponc}

Hence, a closed Poncelet chain $P_1\ldots P_n$ forms an $n$-gon simultaneously inscribed in the first conic $\mathcal C_1$ and circumscribed around the second conic $\mathcal C_2$. If one such a polygon exists, then there exists infinitely many (see Fig.~\ref{pic1} (a)).

\begin{figure}[h]
{\footnotesize
\begin{center}

\begin{minipage}[H]{0.49\linewidth}
\begin{center}

\definecolor{ffqqqq}{rgb}{1,0,0}
\definecolor{qqqqff}{rgb}{0,0,1}
\begin{tikzpicture}[line cap=round,line join=round,>=triangle 45,x=1.0cm,y=1.0cm,scale=0.35]
\clip(-1.77,-6.89) rectangle (15.84,7.83);
\draw [rotate around={4.47:(7.03,0.54)},line width=1pt,color=qqqqff] (7.03,0.54) ellipse (7.73cm and 6.68cm);
\draw [rotate around={34.27:(8.82,0.22)},line width=1pt,color=ffqqqq] (8.82,0.22) ellipse (2.43cm and 1.83cm);
\draw (0.71,4.25)-- (14.73,1.13);
\draw (14.73,1.13)-- (2.53,-4.98);
\draw (2.53,-4.98)-- (11.04,6.33);
\draw (11.04,6.33)-- (11.12,-5.05);
\draw (11.12,-5.05)-- (0.72,4.26);
\draw [dash pattern=on 3pt off 3pt] (8.53,7.12)-- (13.03,-3.55);
\draw [dash pattern=on 3pt off 3pt] (-0.65,1.06)-- (14.32,2.87);
\draw [dash pattern=on 3pt off 3pt] (-0.65,1.06)-- (13.03,-3.55);
\draw [dash pattern=on 3pt off 3pt] (8.53,7.12)-- (4.96,-5.94);
\draw [dash pattern=on 3pt off 3pt] (14.32,2.87)-- (4.96,-5.94);
\draw (-1.24,-2.37) node[anchor=north west] {$\mathcal C_1$};
\draw (7.08,0.43) node[anchor=north west] {$\mathcal C_2$};
\fill [color=black] (0.71,4.25) circle (1.0pt);
\draw[color=black] (0.41,4.77) node {$P_1$};
\fill [color=black] (14.73,1.13) circle (1.0pt);
\draw[color=black] (15.4,1.19) node {$P_2$};
\fill [color=black] (2.53,-4.98) circle (1.0pt);
\draw[color=black] (2.34,-5.51) node {$P_3$};
\fill [color=black] (11.04,6.33) circle (1.0pt);
\draw[color=black] (11.5,6.68) node {$P_4$};
\fill [color=black] (11.12,-5.05) circle (1.0pt);
\draw[color=black] (11.59,-5.47) node {$P_5$};
\end{tikzpicture}

\end{center}

\end{minipage}
\hfill 
\begin{minipage}[h]{0.49\linewidth}
\begin{center}

\definecolor{ffqqqq}{rgb}{1,0,0}
\definecolor{qqzzqq}{rgb}{0,0.6,0}
\definecolor{qqqqff}{rgb}{0,0,1}
\definecolor{zzttqq}{rgb}{0.6,0.2,0}
\definecolor{uququq}{rgb}{0.25,0.25,0.25}
\begin{tikzpicture}[line cap=round,line join=round,>=triangle 45,x=1.0cm,y=1.0cm,scale=0.55]
\clip(2.05,-3.96) rectangle (13.71,5.12);
\fill[color=zzttqq,fill=zzttqq,fill opacity=0.1] (8.53,0.14) -- (9.12,-1.4) -- (10.15,-0.12) -- cycle;
\draw [shift={(2.98,-3.02)},color=qqqqff,fill=qqqqff,fill opacity=0.1] (0,0) -- (0:1.48) arc (0:14.82:1.48) -- cycle;
\draw [shift={(2.98,-3.02)},color=qqqqff,fill=qqqqff,fill opacity=0.1] (0,0) -- (14.82:1.27) arc (14.82:29.64:1.27) -- cycle;
\draw [shift={(2.98,-3.02)},color=qqqqff,fill=qqqqff,fill opacity=0.1] (0,0) -- (29.64:1.48) arc (29.64:44.45:1.48) -- cycle;
\draw [shift={(10.4,4.26)},color=qqzzqq,fill=qqzzqq,fill opacity=0.1] (0,0) -- (-135.55:1.27) arc (-135.55:-114.38:1.27) -- cycle;
\draw [shift={(10.4,4.26)},color=qqzzqq,fill=qqzzqq,fill opacity=0.1] (0,0) -- (-114.38:1.06) arc (-114.38:-93.21:1.06) -- cycle;
\draw [shift={(10.4,4.26)},color=qqzzqq,fill=qqzzqq,fill opacity=0.1] (0,0) -- (-93.21:1.27) arc (-93.21:-72.04:1.27) -- cycle;
\draw [shift={(12.76,-3.02)},color=ffqqqq,fill=ffqqqq,fill opacity=0.1] (0,0) -- (107.96:1.27) arc (107.96:131.97:1.27) -- cycle;
\draw [shift={(12.76,-3.02)},color=ffqqqq,fill=ffqqqq,fill opacity=0.1] (0,0) -- (131.97:1.06) arc (131.97:155.99:1.06) -- cycle;
\draw [shift={(12.76,-3.02)},color=ffqqqq,fill=ffqqqq,fill opacity=0.1] (0,0) -- (155.99:1.27) arc (155.99:180:1.27) -- cycle;
\draw [line width=1pt] (2.98,-3.02)-- (10.4,4.26);
\draw [line width=1pt] (10.4,4.26)-- (12.76,-3.02);
\draw [line width=1pt] (2.98,-3.02)-- (12.76,-3.02);
\draw (10.4,4.26)-- (9.99,-3.02);
\draw (10.4,4.26)-- (7.1,-3.02);
\draw (2.98,-3.02)-- (11.24,1.68);
\draw (2.98,-3.02)-- (11.99,-0.64);
\draw (12.76,-3.02)-- (8.17,2.08);
\draw (12.76,-3.02)-- (6.03,-0.02);
\draw [color=zzttqq] (8.53,0.14)-- (9.12,-1.4);
\draw [color=zzttqq] (9.12,-1.4)-- (10.15,-0.12);
\draw [color=zzttqq] (10.15,-0.12)-- (8.53,0.14);
\draw [shift={(10.4,4.26)},color=qqzzqq] (-135.55:1.27) arc (-135.55:-114.38:1.27);
\draw [shift={(10.4,4.26)},color=qqzzqq] (-135.55:1.16) arc (-135.55:-114.38:1.16);
\draw [shift={(10.4,4.26)},color=qqzzqq] (-114.38:1.06) arc (-114.38:-93.21:1.06);
\draw [shift={(10.4,4.26)},color=qqzzqq] (-114.38:0.95) arc (-114.38:-93.21:0.95);
\draw [shift={(10.4,4.26)},color=qqzzqq] (-93.21:1.27) arc (-93.21:-72.04:1.27);
\draw [shift={(10.4,4.26)},color=qqzzqq] (-93.21:1.16) arc (-93.21:-72.04:1.16);
\draw [shift={(12.76,-3.02)},color=ffqqqq] (107.96:1.27) arc (107.96:131.97:1.27);
\draw [shift={(12.76,-3.02)},color=ffqqqq] (107.96:1.16) arc (107.96:131.97:1.16);
\draw [shift={(12.76,-3.02)},color=ffqqqq] (107.96:1.06) arc (107.96:131.97:1.06);
\draw [shift={(12.76,-3.02)},color=ffqqqq] (131.97:1.06) arc (131.97:155.99:1.06);
\draw [shift={(12.76,-3.02)},color=ffqqqq] (131.97:0.95) arc (131.97:155.99:0.95);
\draw [shift={(12.76,-3.02)},color=ffqqqq] (131.97:0.85) arc (131.97:155.99:0.85);
\draw [shift={(12.76,-3.02)},color=ffqqqq] (155.99:1.27) arc (155.99:180:1.27);
\draw [shift={(12.76,-3.02)},color=ffqqqq] (155.99:1.16) arc (155.99:180:1.16);
\draw [shift={(12.76,-3.02)},color=ffqqqq] (155.99:1.06) arc (155.99:180:1.06);
\draw [dash pattern=on 2pt off 2pt] (8.53,0.14)-- (12.76,-3.02);
\draw [dash pattern=on 2pt off 2pt] (2.98,-3.02)-- (10.15,-0.12);
\draw [dash pattern=on 2pt off 2pt] (10.4,4.26)-- (9.12,-1.4);
\fill [color=uququq] (2.98,-3.02) circle (1.0pt);
\draw[color=uququq] (2.77,-3.18) node {$A$};
\fill [color=uququq] (10.4,4.26) circle (1.0pt);
\draw[color=uququq] (10.56,4.51) node {$B$};
\fill [color=uququq] (12.76,-3.02) circle (1.0pt);
\draw[color=uququq] (12.97,-3.18) node {$C$};
\fill [color=uququq] (11.24,1.68) circle (1.0pt);
\fill [color=uququq] (11.99,-0.64) circle (1.0pt);
\fill [color=uququq] (7.1,-3.02) circle (1.0pt);
\fill [color=uququq] (9.99,-3.02) circle (1.0pt);
\fill [color=uququq] (8.17,2.08) circle (1.0pt);
\fill [color=uququq] (6.03,-0.02) circle (1.0pt);
\fill [color=uququq] (10.15,-0.12) circle (1.0pt);
\draw[color=uququq] (10.61,0.07) node {$U_1$};
\fill [color=uququq] (9.12,-1.4) circle (1.0pt);
\draw[color=uququq] (9.25,-1.86) node {$V_1$};
\fill [color=uququq] (8.53,0.14) circle (1.0pt);
\draw[color=uququq] (8.14,0.44) node {$W_1$};
\end{tikzpicture}

\end{center}
\end{minipage}
\begin{minipage}[h]{1\linewidth}
\begin{tabular}{p{0.49\linewidth}p{0.49\linewidth}}
\vskip 0.05mm \centering (a) & \vskip 0.05mm \centering (b)
\end{tabular}
\end{minipage}

\caption{(a) The closed in $5$ steps Poncelet chain $P_1\ldots P_5$ for the nested ellipses $\mathcal C_1$ and $\mathcal C_2$; (b) Morley's trisector theorem.}
\label{pic1}
\end{center}
}
\end{figure}
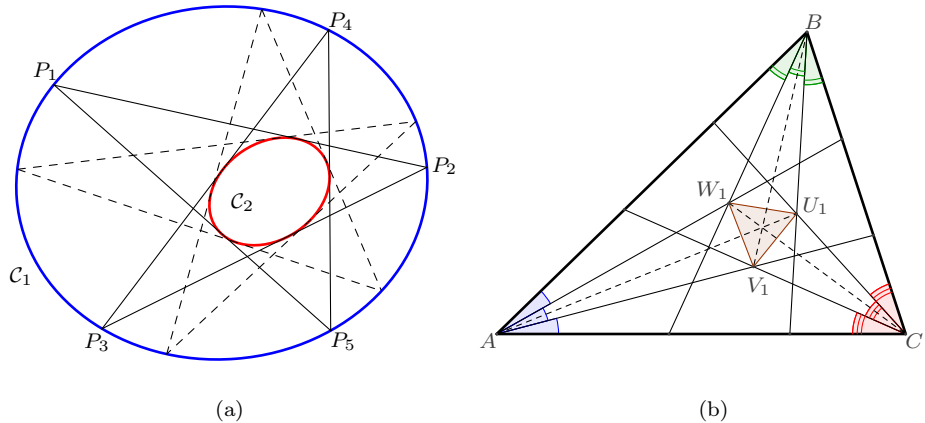

Another undoubtedly ``immortal'' geometrical result was born in 1899 with the help of Frank Morley, and is known now as \textit{Morley's theorem} (see the classical overview~\cite{OB78}, and some resent proofs~\cite{C98},~\cite{Con06}). Overlooked by thousands of mathematicians before Morley, this theorem presents the following astonishing fact from the geometry of triangles.

\begin{morl}
For any triangle $\triangle ABC$ the three points $U_1$, $V_1$, $W_1$ of intersection of the adjacent angle \emph{trisectors}, that is rays emanating from the vertexes of $\triangle ABC$ and dividing the corresponding angles into three equal parts, form an equilateral triangle $\triangle U_1V_1W_1$, called the \emph{Morley triangle} (see Fig.~\ref{pic1} (b)).
\end{morl}

Further we shall see how the two facts described above can meet each other, namely, how Poncelet's porism appears in the construction of the Morley triangle. For such a purpose in the next section we present a useful criterion for six points to be \textit{conconic}, that is lie on a single conic.

\section{When six points are conconic?}

It is well known (see~\cite{AkZ07}, or~\cite{Ber87}) that any five distinct points (or lines) in general position on the plane uniquely determine a non-degenerate conic passing through (respectively, touching) these points (respectively, lines). However, if three or more points are collinear (respectively, three or more lines intersect in the same point), then the conic can degenerate and can even fail to be unique. Everything said above also holds on the \textit{projective plane}, that is the Euclidean plane equipped with the so-called \textit{line at infinity} consisting of \textit{points at infinity} in which intersect all families of parallel lines. For the detailed exposition of projective geometry we send the reader to~\cite{Cox03}.     

Since five points determine a conic, the right question to ask is when a sixth point lie on a single conic. We now state one of the possible answers to this question. Recall that a ray from a vertex of a triangle through a point on the opposite side (or its continuation) called a \textit{cevian}.

\begin{theorem}
\label{criterion}
Let $AA_1$, $BB_1$, $CC_1$ be three cevians of $\triangle ABC$ mutually meeting at the points $X_1 = BB_1 \cap CC_1$, $Y_1 = AA_1 \cap CC_1$, and $Z_1 = AA_1 \cap BB_1$. Suppose also that $AA_2$, $BB_2$, $CC_2$ is an another triple of cevians meeting at the similarly constructed points $X_2$, $Y_2$, and $Z_2$. If we further put $U_1 := BB_1 \cap CC_2$, $V_1 := CC_1 \cap AA_2$, and $W_1 := AA_1 \cap BB_2$ (see Fig.~\ref{pic2} (a)), then the following conditions are equivalent:
\begin{enumerate}
\item
\label{outer6}
$A_1$, $A_2$, $B_1$, $B_2$, $C_1$, $C_2$ lie on a single conic $\mathcal C_0$;

\item
\label{inner6}
$X_1$, $X_2$, $Y_1$, $Y_2$, $Z_1$, $Z_2$ lie on a single conic $\mathcal C_1$;

\item
\label{brian}
$AA_1$, $AA_2$, $BB_1$, $BB_2$, $CC_1$, and $CC_2$ touch the same conic $\mathcal C_2$;

\item
\label{2ndMC}
$AU_1$, $BV_1$, and $CW_1$ intersect in a single point.

\end{enumerate}
 
\end{theorem}

\begin{remark}
Some implications of Theorem~\ref{criterion} seem to be folkloric statements rediscovered independently by several authors. Although the author of the present note was aware about the statement of the theorem in 2007, the equivalence of \ref{outer6}, \ref{brian}, and \ref{2ndMC} first appeared in~\cite{Dol11}. The equivalence of \ref{outer6}, \ref{inner6}, and \ref{brian} was proved in~\cite{Sz12} using barycentric coordinates. Recently in~\cite{Bar14} D.~Barali\'c also proved the equivalence of \ref{outer6}, \ref{inner6}, and \ref{brian} by using a smart application of Carnot's theorem. Our approach will be more straightforward then in the previous works.
\end{remark}

\begin{figure}[h]
{\tiny
\begin{center}

\begin{minipage}[H]{0.49\linewidth}
\begin{center}

\definecolor{qqqqff}{rgb}{0,0,1}
\definecolor{ffqqqq}{rgb}{1,0,0}
\definecolor{qqzzqq}{rgb}{0,0.6,0}
\definecolor{uququq}{rgb}{0.25,0.25,0.25}
\begin{tikzpicture}[line cap=round,line join=round,>=triangle 45,x=1.0cm,y=1.0cm, scale = 0.5]
\clip(-1.83,-3.5) rectangle (10.39,6.82);
\draw [line width=0.7pt] (-1.2,-2.66)-- (8.18,6.28);
\draw [line width=0.7pt] (8.18,6.28)-- (9.66,-2.66);
\draw [line width=0.7pt] (9.66,-2.66)-- (-1.2,-2.66);
\draw [rotate around={-158.33:(5.89,0.28)},line width=0.7pt,color=qqzzqq] (5.89,0.28) ellipse (4.06cm and 3.08cm);
\draw (9.66,-2.66)-- (4.98,3.23);
\draw (9.66,-2.66)-- (2.06,0.45);
\draw (-1.2,-2.66)-- (8.74,2.91);
\draw (-1.2,-2.66)-- (9.34,-0.72);
\draw (8.18,6.28)-- (6.82,-2.66);
\draw (8.18,6.28)-- (3.6,-2.66);
\draw [rotate around={-159.44:(6.42,0.18)},line width=0.7pt,color=ffqqqq] (6.42,0.18) ellipse (3.27cm and 1.84cm);
\draw [rotate around={103.28:(6.22,0.03)},line width=0.7pt,color=qqqqff] (6.22,0.03) ellipse (1.3cm and 0.94cm);
\draw [dash pattern=on 2pt off 2pt] (5.53,1.11)-- (9.66,-2.66);
\draw [dash pattern=on 2pt off 2pt] (-1.2,-2.66)-- (7.28,0.34);
\draw [dash pattern=on 2pt off 2pt] (8.18,6.28)-- (6.28,-1.28);
\draw (2.52,3) node[anchor=north west] {$\mathcal C_0$};
\draw (3.8,2) node[anchor=north west] {$\mathcal C_1$};
\draw (5.2,0.7) node[anchor=north west] {$\mathcal C_2$};
\fill [color=uququq] (-1.2,-2.66) circle (1.0pt);
\draw (-1.44,-2.71) node {$A$};
\fill [color=uququq] (8.18,6.28) circle (1.0pt);
\draw (8.33,6.6) node {$B$};
\fill [color=uququq] (9.66,-2.66) circle (1.0pt);
\draw (10,-2.74) node {$C$};
\fill [color=uququq] (2.06,0.45) circle (1.0pt);
\draw (1.73,0.73) node {$C_1$};
\fill [color=uququq] (4.98,3.23) circle (1.0pt);
\draw (4.75,3.58) node {$C_2$};
\fill [color=uququq] (8.74,2.91) circle (1.0pt);
\draw (9.25,3.05) node {$A_1$};
\fill [color=uququq] (9.34,-0.72) circle (1.0pt);
\draw (9.81,-0.75) node {$A_2$};
\fill [color=uququq] (3.6,-2.66) circle (1.0pt);
\draw (3.56,-3.1) node {$B_2$};
\fill [color=uququq] (6.82,-2.66) circle (1.0pt);
\draw (7.04,-3.1) node {$B_1$};
\fill [color=uququq] (5.98,1.97) circle (1.5pt);
\draw (5.9,2.7) node {$X_2$};
\fill [color=uququq] (7.57,2.25) circle (1.5pt);
\draw (7.92,1.9) node {$Z_1$};
\fill [color=uququq] (8.27,-0.91) circle (1.5pt);
\draw (8.35,-0.39) node {$Y_2$};
\fill [color=uququq] (6.99,-1.57) circle (1.5pt);
\draw (7.3,-2.05) node {$X_1$};
\fill [color=uququq] (4.1,-1.68) circle (1.5pt);
\draw (4.3,-2.15) node {$Z_2$};
\fill [color=uququq] (3.38,-0.09) circle (1.5pt);
\draw (3.2,0.34) node {$Y_1$};
\fill [color=uququq] (5.53,1.11) circle (1.0pt);
\draw (5.2,1.29) node {{\tiny $W_1$}};
\fill [color=uququq] (7.28,0.34) circle (1.0pt);
\draw (7.74,0.44) node {{\tiny $U_1$}};
\fill [color=uququq] (6.28,-1.28) circle (1.0pt);
\draw (6.29,-1.56) node {{\tiny $V_1$}};
\draw (6.75,-0.3) node {{\tiny $O$}};
\end{tikzpicture}

\end{center}

\end{minipage}
\hfill 
\begin{minipage}[h]{0.49\linewidth}
\begin{center}

\definecolor{uququq}{rgb}{0.25,0.25,0.25}
\begin{tikzpicture}[line cap=round,line join=round,>=triangle 45,x=1.0cm,y=1.0cm, scale = 0.55]
\clip(1.74,-5.17) rectangle (12.83,4.04);
\draw [line width=0.8pt,domain=1.74:12.83] plot(\x,{(-13.8-0*\x)/7.42});
\draw [line width=0.8pt] (5.04,-5.17) -- (5.04,4.04);
\draw (4.45,-1.9) node[anchor=north west] {$\bar{A}$};
\draw (3.4,0.2) node[anchor=north west] {$\bar{C}_1\!(0,\!1)$};
\draw (6.83,-1.9) node[anchor=north west] {$\bar B_2\!(1,\!0)$};
\draw (4.56,4) node[anchor=north west] {$\bar b$};
\draw (12.2,-1.85) node[anchor=north west] {$\bar c$};
\draw [domain=5.04:12.83] plot(\x,{(-35.12--4.2*\x)/7.5});
\draw [domain=5.04:12.83] plot(\x,{(-37.68--5.92*\x)/4.21});
\draw (9.9,-1.9) node[anchor=north west] {$\bar B_1\!(b_1\!,\!0)$};
\draw (3.2,2.6) node[anchor=north west] {$\bar C_2\!(0,\!c_2\!)$};
\draw [domain=1.74:12.83] plot(\x,{(-56.07--4.2*\x)/7.5});
\draw [domain=1.74:12.83] plot(\x,{(-21.11--5.92*\x)/4.21});
\draw (12.2,3.15) node[anchor=north west] {$\bar a_2$};
\draw (8.18,4) node[anchor=north west] {$\bar a_1$};
\draw (5.0,-4.55) node[anchor=north west] {$\bar P (0,\!p)$};
\draw (1.55,-2.4) node[anchor=north west] {$\bar Q (q,\!0)$};
\draw (10.0,2.8) node[anchor=north west] {$\bar U_1 \!(b_1\!,\! c_2)$};
\draw (5.02,1.88) node[anchor=north west] {$\bar W_1\!\left(\!1,\!\frac{c_2}{q}\!\right)$};
\draw (8.05,0.1) node[anchor=north west] {$\bar V_1\!\left(\!\frac{b_1}{p}\!,\! 1\!\right)$};
\draw (10.02,-1.86) -- (10.02,4.04);
\draw (7.26,-1.86) -- (7.26,4.04);
\draw [domain=5.04:12.827946396871424] plot(\x,{(--1.43-0*\x)/3.96});
\draw [domain=5.04:12.827946396871424] plot(\x,{(--10.34-0*\x)/4.98});
\draw [dash pattern=on 2pt off 2pt] (5.04,-1.86)-- (10.02,2.07);
\draw [dash pattern=on 2pt off 2pt] (7.26,1.26)-- (9,1.26);
\draw [dash pattern=on 2pt off 2pt] (9,1.26)-- (9,0.36);
\draw (8.55,1.85) node[anchor=north west] {$\bar O$};
\draw [dash pattern=on 2pt off 2pt] (5.04,0.36)-- (7.26,-1.86);
\draw [dash pattern=on 2pt off 2pt] (2.24,-1.86)-- (5.04,-4.65);
\fill [color=uququq] (5.04,-1.86) circle (1.5pt);
\fill [color=uququq] (7.26,-1.86) circle (1.5pt);
\fill [color=uququq] (5.04,0.36) circle (1.5pt);
\fill [color=uququq] (10.02,-1.86) circle (1.5pt);
\fill [color=uququq] (5.04,2.07) circle (1.5pt);
\fill [color=uququq] (2.24,-1.86) circle (1.5pt);
\fill [color=uququq] (5.04,-4.65) circle (1.5pt);
\fill [color=uququq] (10.02,2.07) circle (1.5pt);
\fill [color=uququq] (7.26,1.26) circle (1.5pt);
\fill [color=uququq] (9,0.36) circle (1.5pt);
\fill [color=uququq] (9,1.26) circle (1.5pt);
\end{tikzpicture}

\end{center}
\end{minipage}
\begin{minipage}[h]{1\linewidth}
\begin{tabular}{p{0.49\linewidth}p{0.49\linewidth}}
\vskip 0.05mm \centering {\footnotesize (a)} & \vskip 0.05mm \centering {\footnotesize (b)}
\end{tabular}
\end{minipage}

\caption{(a) The configuration of Theorem~\ref{criterion}; (b) The configuration of Theorem~\ref{criterion} after a suitable projective transformation.}
\label{pic2}
\end{center}
}
\end{figure}
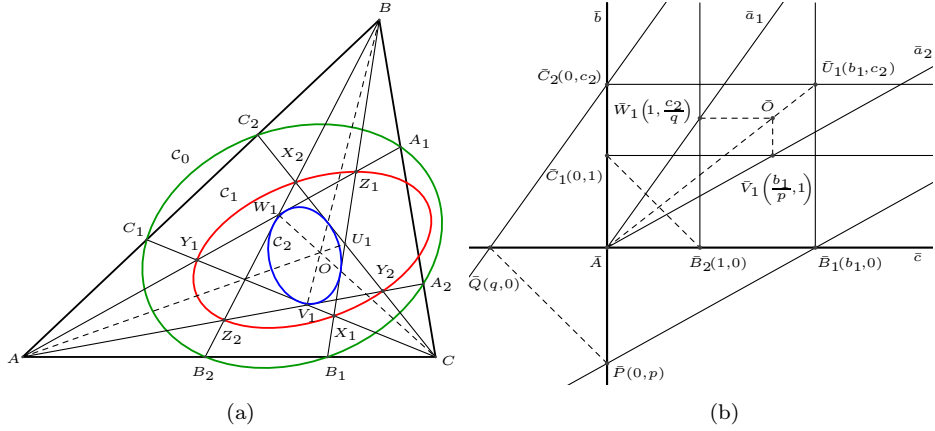


\begin{proof}

Let us start with showing the equivalence of~\ref{outer6} and~\ref{2ndMC}. Since the construction of the theorem is invariant under projective transformations, we can apply a particular one to help us. More precisely, let us make the transformation that maps the line $BC$ to the line at infinity. This is a standard trick in projective geometry that simplifies the whole picture. If we denote the resulting image of an arbitrary point $X$ as $\bar X$, then $\bar B$, $\bar C$, $\bar A_1$, $\bar A_2$ are the points at infinity (see Fig.~\ref{pic2} (b)). Put $\bar b$, $\bar c$, $\bar a_1$, and $\bar a_2$ for the images of the lines, respectively, $AB$, $AC$, $AA_1$, and $AA_2$. Furthermore, we can assume that the lines $\bar b$ and $\bar c$ are perpendicular, and $\bar A \bar C_1 = \bar A \bar B_2 = 1$.

Introduce the coordinate system with the origin at $\bar A$ and axises $\bar b$ and $\bar c$, and set the coordinates for the points as $\bar B_1 (b_1, 0)$, $\bar C_2 (0, c_2)$, $\bar P (0, p)$, $\bar Q (q, 0)$ with $P := AB \cap A_2 B_1$ and $Q := AC \cap A_1 C_2$ (possibly being points at infinity). 

Since the points $A_1$ and $A_2$ were sent to infinity, we get $\bar a_1 \parallel \bar Q \bar C_2$ and $\bar a_2 \parallel \bar B_1 \bar P$. Using these relations, it is straightforward to show that the points $\bar U_1$, $\bar V_1$, and $\bar W_1$ have the following coordinates: $\bar U_1 = (b_1, c_2)$, $\bar V_1 = (b_1/p, 1)$, and $\bar W_1 = (1, c_2/q)$. 

If $O = BV_1 \cap CW_1$, then the condition~\ref{2ndMC} is equivalent to $O \in AU_1$, or $\bar O \in \bar A \bar U_1$. It is easy to see that $\bar O = (b_1/p, c_2/q)$. 

Therefore $\bar O \in \bar A \bar U_1$ if and only if $p = q$. This equality is equivalent to the fact that $\bar Q \bar P \parallel \bar C_1 \bar B_2$. From here we can get the equivalence of~\ref{outer6} and~\ref{2ndMC} by incorporating \textit{Pascal's theorem} stating that \textit{on the projective plane the six vertexes of a hexagon lie on the same conic if and only if the three points of intersection of the opposite sides of the hexagon are collinear} (see~\cite{AkZ07}, or~\cite{Cox03}). Using this fundamental fact, $\bar Q \bar P \parallel \bar C_1 \bar B_2$ is, in turn, equivalent to the fact that $A_1$, $A_2$, $B_1$, $B_2$, $C_1$, and $C_2$ lie on a single conic $\mathcal C_0$. The equivalence $\ref{outer6}\Leftrightarrow\ref{2ndMC}$ is proved.

The equivalence $\ref{brian}\Leftrightarrow\ref{2ndMC}$ is the corollary of an another fundamental fact, namely, \textit{Brianchon's theorem} which states that \textit{on the projective plane the sidelines of a hexagon are tangent to a single conic if and only if three lines joining the opposite vertexes meet at a single point} (see~\cite{AkZ07}, \cite{Cox03}). 

Finally, let us prove the equivalence $\ref{inner6}\Leftrightarrow\ref{brian}$. Suppose~\ref{inner6} holds. Consider the conic $\mathcal C_2'$ that simultaneously touches all the cevians except $CC_2$. We know that $\mathcal C_2'$ is uniquely defined. Hence, we have two conics $\mathcal C_1$ and $\mathcal C_2'$ and the triangle $X_1Y_1Z_1$ whose vertexes lie on $\mathcal C_1$ and whose sides are tangent to $\mathcal C_2'$. Therefore, by Poncelet's theorem, starting from any point on $\mathcal C_1$ the corresponding Poncelet chain always closes up in three steps. Let us pick the point $X_2$. Since the Poncelet chain starting at $X_2$ must close in three steps, we get that $Y_2Z_2$ is also tangent to $\mathcal C_2'$. Thus $\mathcal C_2' = \mathcal C_2$, and the implication $\ref{inner6} \Rightarrow \ref{brian}$ is proved. The reverse implication can be proved by reversing the arguments above. We note that this shortcut with Poncelet's theorem also appeared in~\cite{Bar14} (see~\cite{HH14} for exploiting this idea to give a proof for Poncelet's theorem).

The last thing we should note here is that we omitted possible degenerate cases since the computation in such cases can be easily restored by the reader. Theorem~\ref{criterion} is proved.
\end{proof}

\section{What about Morley?}

Now we will see how both celebrated facts mentioned in the introduction can meet each other. Again, let $U_1$, $V_1$, and $W_1$ be intersection points of the adjacent angle trisectors of $\triangle ABC$ from the Morley configuration. It is known that the lines $AU_1$, $BV_1$, $CW_1$ intersect in a single point called the \textit{second Morley center} (with the \textit{first Morley center} being just the center of the Morley triangle) (see~\cite{K94}). Thus, using conditions \ref{inner6} and \ref{brian} of Theorem~\ref{criterion}  in the configuration of Morley's theorem we can construct two conics $\mathcal C_1$ and $\mathcal C_2$ for which there exist two closed Poncelet chains $X_1Y_1Z_1$ and $X_2Y_2Z_2$. And Poncelet's theorem tells us that there are infinitely many such triangles for each starting point on $\mathcal C_1$! 

Hence, in the Morley configuration one can find two quite ``rare'' objects~-- a pair of conics for which Poncelet chains are closing up! For an arbitrary given pair of conics such a conclusion is not always the case. There is still much discussion about simple, like \textit{Fuss's formulas}, and powerful, like general \textit{Cayley's theorem}, criteria ensuring existence of a closed Poncelet chain for a given pair of conics (see~\cite[Ch. 10]{Fl09}, \cite{DrR11} for further details). 
	
\section{Closing remarks}

In fact, Theorem~\ref{criterion} in Morley's case gives us more, namely, that the intersections of the angular trisectors with the sides of the triangle belong to a single conic, which is a consequence of condition~\ref{outer6} of Theorem~\ref{criterion}.

Surprisingly, the same holds for any triple of \textit{isogonally conjugate} cevians, that is cevians symmetric with respect to the corresponding angle bisectors. More precisely, one can get the following

\begin{theorem}
\label{isog}
If $AA_1, AA_2$, $BB_1, BB_2$, and $CC_1, CC_2$ are three pairs of isogonally conjugate cevians of $\triangle ABC$, then condition~\ref{outer6} of Theorem~\ref{criterion} holds.
\end{theorem}

This result can be easily proved using condition~\ref{brian} of Theorem~\ref{criterion} via straightforward calculation in coordinates; see also~\cite{Dol11}. 

The fact dual to Theorem~\ref{isog} is the following 

\begin{theorem}
\label{isotom}
Suppose $AA_1, AA_2$, $BB_1, BB_2$, and $CC_1, CC_2$ are three pairs of \emph{isotomically} conjugate cevians, that is cevians for which the pairs of points $A_1$, $A_2$, $B_1$, $B_2$, and $C_1$, $C_2$ are symmetric with respect to the midpoints of the corresponding sides; then condition~\ref{outer6} of Theorem~\ref{criterion} holds.
\end{theorem}

On account of projective equivalence of the mentioned types of conjugacy we refer the reader to~\cite{Ak12}. Theorem~\ref{isotom} may be proved by a direct computation using Theorem~\ref{criterion}, or can be obtained as a corollary of Carnot's theorem (see~\cite{AkZ07}).

Finally, we can get another corollary of Theorem~\ref{criterion} ``for free'' (see~\cite{Gale96} and~\cite{Gib06} for an alternative approach).

\begin{theorem}
Let $P_1$ and $P_2$ be any two points inside $\triangle ABC$. If for any $i \in \{1,2\}$ cevians $AA_i$, $BB_i$, $CC_i$ are passing through $P_i$, then condition~\ref{outer6} of Theorem~\ref{criterion} holds.
\end{theorem}

\begin{proof}
Since in this case $X_i = Y_i = Z_i = P_i$ for any $i \in \{1,2\}$, condition~\ref{inner6} of Theorem~\ref{criterion} holds trivially for a degenerate conic $\mathcal C_1$ being the (doubly covered) line passing through $P_1$ and $P_2$.
\end{proof}

\section*{Acknowledgment}
The author would like to thank Yuliya V. Eremenko for her guidance during the preparation for the Regional scientific contest in 2007, during which a part of this work was done, and for all her help ever since.

\medskip

\textsc{Kostiantyn Drach}:
\textit{V.N. Karazin Kharkiv National University, Geometry Department,}

\textit{Svobody Sq. 4, 61022, Ukraine; and}

\textit{Sumy State University, Department of Mathematical Analysis and Optimization,}

\textit{Rimskogo--Korsakova str. 2, 40007, Ukraine}

\textit{E-mail:} \textsf{drach@karazin.ua}


\begin{thebibliography}{1}
\bibitem{Ak12} A. V. Akopyan, \textit{Conjugation of lines with respect to a triangle}, Journal of Classical Geometry, \textbf{1} (2012) 23--31.

\bibitem{AkZ07} A. V. Akopyan, A. A. Zaslavsky, \textit{Geometry of conics}, Amer. Math. Soc., Providence, RI, 2007. 

\bibitem{Bar14} D. Barali\'c, A Short Proof of the Bradley Theorem, \textit{Amer. Math. Monthly} (to appear, also available at arXiv:1308.6144 [math.MG])


\bibitem{BB96} W. Barth, T. Bauer, Poncelet theorems, \textit{Expos. Math.}, \textbf{14} (1996) 125--144.

\bibitem{Ber87} M. Berger, \textit{Geometry}, Springer-Verlag, Berlin, 1987.

\bibitem{C98} A. Connes, \textit{A new proof of Morley's theorem}, Publications Math\'ematiques de l'IH\'ES, \textbf{S88} (1998), 43-46 

\bibitem{Con06} J. H. Conway, The power of mathematics, in \textit{Power}, (eds. A. Blackwell, D. MacKay), Cambridge University Press, 2006; 36--50.

\bibitem{Cox03} H. S. M. Coxeter, \textit{Projective geometry}, 2nd ed., Springer-Verlag, New York, 2003.

\bibitem{Dol11} P. Dolgirev, On lines in a triangle tangent to a conic, \textit{Matematicheskoye Obrazovaniye}, \textbf{59-60}: 3-4 (2011) 15--24 (in Russian) (also available at  arXiv:1101.3283 [math.MG]).

\bibitem{DrR11} V. Dragovi\'c, M. Radnovi\'c, \textit{Poncelet porisms and beyond: integrable billiards, hyperelliptic Jacobians and pencils of quadrics}, Frontiers in Mathematics, Springer Birkh{\"a}user, Basel, 2011. 

\bibitem{Dui10} J. J. Duistermaat, \textit{Discrete integrable systems: QRT maps and elliptic surfaces}, Springer, 2010.

\bibitem{Fl09} L. Flatto, \textit{Poncelet's theorem}, American Mathematical Society, Providence, RI, 2009. 

\bibitem{Gale96} D.Gale, Triangles and Proofs, \textit{Math. Intell.}, \textbf{18}:1 (1996) 31--34. 

\bibitem{Gib06} B. Gibert, Bicevian Conics and CPCC Cubics, avaliable at {http://bernard.gibert.pagesperso-orange.fr/}.

\bibitem{HH14} L. Halbeisen, N. Hungerb\''uhler, A Simple Proof of Poncelet’s Theorem, \textit{Amer. Math. Monthly} (to appear)

\bibitem{K94} C. Kimberling, \textit{Central points and central lines in the plane of a triangle}, Math. Mag. \textbf{67} (1994), 163--187.

\bibitem{LT07} M. Levi, S. Tabachnikov, The Poncelet grid and billiards in ellipses, \textit{Amer. Math. Monthly}, \textbf{114} (2007), 895--908.

\bibitem{OB78} C. O. Oakley, J. C. Baker, The Morley trisector theorem, \textit{Amer. Math. Monthly}, \textbf{85} (1978) 737--745.

\bibitem{Sz12} Z. Szilasi,\textit{Two applications of the theorem of Carnot}, Annales Mathematicae et Informaticae, \textbf{40} (2012), 135--144.

\end{thebibliography}
\end{document}